\documentclass[a4,12pt,reqno]{amsart}
\usepackage{amsthm}
\usepackage{amsmath}
\usepackage{amssymb}
\usepackage{latexsym}
\usepackage{enumerate}
 \newtheorem{theorem}{Theorem}[section]
 \newtheorem{lem}[theorem]{Lemma}
 \newtheorem{prop}[theorem]{Proposition}
 
\theoremstyle{remark} 
 \newtheorem{rem}{Remark}[section]
  
\makeatletter
\@addtoreset{equation}{section}

\makeatother

\begin{document}
\title[Life span of nonlinear Schr\"{o}dinger equation]{
On the life span of the Schr\"{o}dinger equation
with sub-critical power nonlinearity}
\author[H. Sasaki]{
Hironobu Sasaki$^*$ \\
Department of Mathematics and Informatics, Chiba University, 
263--8522, Japan.
}
\thanks{$^*$Supported by Research Fellowships of 
the Japan Society for the Promotion of Science 
for Young Scientists.}
\thanks{$^*$email:
\texttt{sasaki@math.s.chiba-u.ac.jp}} 
\subjclass[2000]{35Q55}
\keywords{Life span; Schr\"{o}dinger equation; 
sub-critical power nonlinearity} 
\maketitle
\begin{abstract}
We discuss the life span of 
the Cauchy problem for 
the one-dimensional Schr\"{o}dinger equation 
with a single power nonlinearity $\lambda |u|^{p-1}u$ 
($\lambda\in\mathbb{C}$, $2\le p<3$) prescribed 
an initial data of the form $\varepsilon\varphi$.
Here, $\varepsilon$ stands for the size of the data. 
It is not difficult to see that the life span $T(\varepsilon)$ 
is estimated by $C_0 \varepsilon^{-2(p-1)/(3-p)}$ 
from below, provided $\varepsilon$ is 
sufficiently small. 
In this paper, 
we consider a more precise estimate for $T(\varepsilon)$ and 
we prove that    
$\liminf_{\varepsilon\to 0}\varepsilon^{2(p-1)/(3-p)}T(\varepsilon)$ 
is larger than some positive constant
expressed only by $p$, $\mathrm{Im}\lambda$ and $\varphi$.
\end{abstract}
\section{Introduction}\label{intro}
This paper is concerned with 
the life span of solutions to
the Cauchy problem 
for the one-dimensional nonlinear Schr\"{o}dinger equation 
\begin{align}\label{IVP}
\left\{
  \begin{array}{ll}
i\partial_t u +\frac{1}{2}\partial_x^2 u=
\lambda|u|^{p-1}u
&\text{in $[0,\infty)\times\mathbb{R}$,}\\
u|_{t=0}=\varepsilon\varphi       
&\text{on $\mathbb{R}$.}\\
  \end{array}
\right. 
\end{align}
Here, 
$u=u(t,x)$ is a complex-valued unknown function, 
$(t,x)\in [0,\infty)\times\mathbb{R}$, 
$i=\sqrt{-1}$, 
$\partial_t=\partial/\partial t$,
$\partial_x=\partial/\partial x$,
$\varepsilon>0$, 
$\varphi$ belongs to some suitable function space,
$\lambda\in\mathbb{C}$ and 
$p>1$. 

In order to give the concrete definition of the life span,
we recall a standard result for (\ref{IVP}): 
If $1<p<5$ and $\lambda\in\mathbb{C}$, 
then (\ref{IVP}) is locally well-posed in $L^2(\mathbb{R})$
(See, e.g., Theorem 4.6.1 in \cite{Caz}). 
That is, for any $\varepsilon>0$ and $\varphi\in L^2(\mathbb{R})$, 
there exists some $T>0$ 
such that (\ref{IVP}) has a unique solution
$u\in C([0,T);L^2(\mathbb{R}))$.
Therefore, 
we can define \textit{the life span} $T(\varepsilon)$ of (\ref{IVP}) by 
\begin{align}\label{def:lifespan}
T(\varepsilon) 
=
\sup\{
T> 0;
\text{ (\ref{IVP}) has a unique solution }  
u \in C([0,T);L^2(\mathbb{R}))\}
\end{align} 
for any 
$p\in (1,5)$, 
$\lambda\in\mathbb{C}$, 
$\varphi\in L^2(\mathbb{R})$ and 
$\varepsilon>0$.
\begin{rem}\label{rem:other def}
We give some equivalent definitions of $T(\varepsilon)$. 
Let 
$H^1(\mathbb{R})$ be the Sobolev space defined by 
$H^1(\mathbb{R})=(1-\Delta)^{-1/2}L^2(\mathbb{R})$.  
For any $p>1$, 
$\lambda\in\mathbb{C}$, $\varepsilon>0$ 
and $\varphi\in H^1(\mathbb{R})$, 
(\ref{IVP}) is locally well-posed in $H^1(\mathbb{R})$, 
so that  
we can define a positive number 
$T^\prime(\varepsilon)$ by 
\begin{align*}
T^\prime(\varepsilon)
=
\sup\{
T> 0;
\text{ (\ref{IVP}) has a unique solution }  
u \in C([0,T);H^1(\mathbb{R}))\}.
\end{align*} 
Furthermore, 
if we use function spaces 
$\Sigma$ and $X(T)$ 
defined by 
(\ref{def:sigma}) and (\ref{def:X(T)}) below, 
respectively, 
then we see that 
for any $\varphi\in \Sigma$, 
we can define a positive number
$T^{\prime\prime}(\varepsilon)$
by 
\begin{align*}
T^{\prime\prime}(\varepsilon)
=
\sup\{
T> 0;
\text{ (\ref{IVP}) has a unique solution }  
u \in X(T)\}.
\end{align*} 
For any $p\in (1,5)$, 
$\lambda\in\mathbb{C}$ and $\varepsilon>0$, 
if 
$\varphi\in H^1(\mathbb{R})$ 
(resp. $\varphi\in \Sigma$), 
then  
$T^\prime(\varepsilon)$ 
(resp. $T^{\prime\prime}(\varepsilon)$) 
is equal to the life span $T(\varepsilon)$ 
(for the proof, 
see, e.g., Theorem 5.2.1 in \cite{Caz}).
\end{rem}

Before treating our problem, 
we mention some known results 
concerned with the life span of 
the Cauchy problem (\ref{IVP}) 
in the case $1<p<5$ and $\lambda\in\mathbb{C}$.

We first focus on the case 
$1<p<5$ and $\mathrm{Im}\,\lambda\le 0$. 
It is well-known that 
(\ref{IVP})  
is $L^2(\mathbb{R})$-sub-critical 
and that 
the time-local solution $u(t)$ to (\ref{IVP}) 
satisfies the a priori estimate 
$\| u(t)\|_{L^2(\mathbb{R})}\le \| \varphi\|_{L^2(\mathbb{R})}$. 
We hence see that 
(\ref{IVP}) is globally well-posed in $L^2(\mathbb{R})$,  
so that  
$T(\varepsilon)=\infty$.

We assume that  
$\varepsilon>0$ is sufficiently small 
and that $\varphi$ belongs to some suitable function space. 
It is clear that we have $T(\varepsilon)=\infty$ 
whenever $3<p<5$. 
Indeed, 
for any $3<p<5$ and $\lambda\in\mathbb{R}$, 
it has been proved  
that the time-local solution $u(t)$ to (\ref{IVP}) 
becomes time-global 
and goes to some free solution like  
$U(t)(\varepsilon\phi_+)$
as $t\to\infty$
(See, e.g., \cite{Gi1997,LinStrauss}), 
where 
$U(t)=\exp(it\Delta/2)$ is the free Schr\"{o}dinger propagator.
We can directly apply such methods to the case 
$3<p<5$ and $\lambda\in\mathbb{C}\setminus\mathbb{R}$.
 
In order to consider the remaining case  
$1<p\le 3$ and $\mathrm{Im}\,\lambda> 0$,  
we review some results of the asymptotic behavior for 
the solution to (\ref{IVP}) in the case 
$1<p\le 3$ and $\lambda\in\mathbb{C}$.
We again assume that  
$\varepsilon>0$ is sufficiently small 
and that $\varphi$ belongs to some suitable function space. 
If $1<p \le 3$, then 
$u(t)$ does not behave like any free solutions as $t\to\infty$. 
In the case $p=3$ (resp. $1<p<3$), 
if $\lambda\in\mathbb{R}$, then 
Hayashi--Naumkin \cite{HaNa}
(resp. Hayashi--Kaikina--Naumkin \cite{HaKaNa})
proved the existence of 
a time-global solution $u(t)$ tending to 
some modified free solution like 
$
\mathcal{F}^{-1}
\exp(i\Theta(t,\xi))
\mathcal{F}
U(t)(\varepsilon\phi_+)
$
as $t\to\infty$.
Here, 
$\mathcal{F}$ is the Fourier transform, 
$\Theta(t,\xi)=\lambda |\phi_+(\xi)|^{p-1}s_p(t)$
with $s_p(t)$ given by
\begin{align*}
s_p(t)=
\left\{
  \begin{array}{ll}
\varepsilon^2 \log t&\text{if $p=3$},    \\
\dfrac{2\varepsilon^{p-1} 
t^{(3-p)/2}}{3-p}&\text{if $1<p<3$}.    \\
  \end{array}
\right.
\end{align*} 
Let a complex-valued function $V(s,\xi)$ solve 
the Cauchy problem for 
a nonlinear ordinary differential equation 
\begin{align}\label{ODE-1}
\left\{
  \begin{array}{ll}
i\partial_s V(s,\xi)=\lambda |V(s,\xi)|^{p-1}V(s,\xi),&
(s,\xi)\in [0,B]\times\mathbb{R},\\
V(0,\xi)=e^{-i\pi/4}\phi_+(\xi),&    
\xi\in\mathbb{R},\\
  \end{array}
\right. 
\end{align}
where $B$ is some positive number.
Then we see that 
the modified free solution 
$
\mathcal{F}^{-1}
\exp(i\Theta(t,\xi))
\mathcal{F}
U(t)(\varepsilon\phi_+)
$
with $1<p\le 3$ is nearly equal to 
the function 
$m_p(t,x)=\varepsilon t^{-1/2}\exp(ix^2/2t)V(s_p(t),x/t)$ 
for sufficiently large $t>0$.
Recently, 
the case $1<p\le 3$ and $\mathrm{Im}\lambda<0$ 
is also studied. 
If $\mathrm{Im}\lambda<0$ and $p=3$ 
(resp. $\mathrm{Im}\lambda<0$ and 
$p$ is smaller than and sufficiently close to 3), 
Shimomura \cite{Shimomura} 
(resp. Kita-Shimomura \cite{KiShi}) showed 
the time-global existence and 
that the solution $u(t)$ behaves like $m_p(t)$  
as $t\to\infty$ 
(see also \cite{KiShi2}).
Furthermore, 
\cite{Shimomura,KiShi,KiShi2} proved  
the time-decay estimate 
\begin{align}\label{est:rapid}
\| u(t)\|_\infty \le 
\left\{
  \begin{array}{ll}
C(1+t)^{-1/(p-1)}       &\text{if $1<p<3$},    \\
C(1+t)^{-1/2}(\log(2+t))^{-1/2}&\text{if $p=3$},    \\
  \end{array}
\right. 
\end{align}  
which shows that $u(t)$ decays more rapidly than 
the corresponding free solution does.
The estimate (\ref{est:rapid}) essentially comes from 
\begin{align*}
\| V(s) \|_\infty \le C
(1-s \mathrm{Im}\lambda)^{-1/(p-1)}, 
\quad s\in [0,\infty).
\end{align*} 
On the other hand, 
in the case $1<p\le 3$ and $\mathrm{Im}\lambda>0$,
the function $V(s)$ blows up at some finite $s>0$.
Therefore, 
we can expect that 
$T(\varepsilon)<\infty$ 
even if $\varepsilon$ is small. 
In fact, 
it is given by Kita \cite{Kita} that 
some blow-up property holds
if   
$p$ and $\lambda$ 
satisfy 
$1<p\le 3$, 
$\mathrm{Im}\lambda>0$ and 
other suitable conditions.

Summarizing the above known results, 
we find that $p=3$ is the critical exponent 
with respect to the asymptotic behavior  
of the local solution to (\ref{IVP}). 
Furthermore, 
in the critical and the sub-critical cases $1<p\le 3$, 
it seems that 
the life span $T(\varepsilon)$ 
is different between the cases  
$\mathrm{Im}\lambda<0$ 
and 
$\mathrm{Im}\lambda>0$.\\

Let us focus on the problem (\ref{IVP}) 
in the sub-critical case 
$1<p<3$ and $\mathrm{Im}\,\lambda>0$.
Our aim of the present paper is to study  
the life span $T(\varepsilon)$. 
In particular, we consider 
the dependence of $T(\varepsilon)$ 
upon $\mathrm{Im}\lambda$. 
It can be easily shown    
that $T(\varepsilon)$ is estimated by 
\begin{align}\label{est:lifespan}
T(\varepsilon)\ge C_0\varepsilon^{-2(p-1)/(3-p)}
\end{align}
for some positive constant $C_0$.
In fact,  
introducing 
the space $\Sigma$ for initial data 
and 
the $X$-norm for solutions defined by  
\begin{align}\label{def:sigma}
\Sigma=
\left\{
\varphi\in L^2(\mathbb{R});
\| \varphi\|_\Sigma 
\equiv
\| \varphi \|_{L^2(\mathbb{R})} 
+
\| \partial_x\varphi \|_{L^2(\mathbb{R})}
+ 
\| x \varphi \|_{L^2(\mathbb{R})}
< \infty 
\right\}
\end{align} 
and 
\begin{align*}
\| u(t)\|_{X}
=
\| u(t)\|_{L^2(\mathbb{R})}
+
\| \partial_x u(t)\|_{L^2(\mathbb{R})}
+\| Ju(t)\|_{L^2(\mathbb{R})}, 
\quad
J=x+it\nabla, 
\end{align*} 
respectively 
and assuming that 
the time-local solution $u(t)$ satisfies 
\begin{align*}
\| u(t)\|_{X}\le 2\varepsilon\|\varphi\|_\Sigma,
\quad\text{$0<t<T$ ($T>0$)},
\end{align*} 
we see from the standard energy inequality that
\begin{align*}
\| u(t)\|_{X}
&\le 
\varepsilon \|\varphi\|_\Sigma +
C_1\int_0^T 
\| u(t)\|_{X}
\| u(t)\|_{L^\infty(\mathbb{R})}^{p-1}
dt\\
&\le 
\varepsilon \|\varphi\|_\Sigma +
C_1\int_0^T
(1+t)^{-(p-1)/2}
\| u(t)\|_{X}^p
dt\\
&\le 
\varepsilon \|\varphi\|_\Sigma +
C_1\varepsilon^p
T^{(3-p)/2}
\|\varphi\|_\Sigma^p
\end{align*}
for $0<t<T$.
Here, 
we have used (\ref{est:J}) below 
in the second inequality
and 
the positive constant $C_1$ depends  
only on $p$ and $|\lambda|$.
Therefore, if $C_0$ 
satisfies 
$C_1 C_0^{(3-p)/2}\|\varphi\|_\Sigma^{p-1}\le 1$,
then we see that 
$u(t)$ with 
$\| u(t)\|_{X}\le 2\varepsilon \|\varphi\|_\Sigma$
exists in 
$0<t<T_0:=C_0\varepsilon^{-2(p-1)/(3-p)}$,
which implies (\ref{est:lifespan}).

Unfortunately, we can not see the dependence of $T(\varepsilon)$ 
upon $\mathrm{Im}\lambda$ only by the proof of (\ref{est:lifespan}).
We hence have to prove a more precise estimate of $T(\varepsilon)$
to see such dependence.  
\subsection{Main result}\label{subsec:main result}
We remark that (\ref{est:lifespan}) is equivalent to 
\[
\liminf_{\varepsilon\to 0}\varepsilon^{2(p-1)/(3-p)}T(\varepsilon)>0.
\]
Our goal of this paper is to show 
a precise lower bound of 
$\liminf_{\varepsilon\to 0}\varepsilon^{2(p-1)/(3-p)}T(\varepsilon)$
and to see the dependence of $T(\varepsilon)$ 
upon $\mathrm{Im}\lambda$.  
To introduce our result, 
we define the Fourier transform $\widehat\phi$ by 
\begin{align*}
\widehat\phi(\xi)
=
\frac{1}{\sqrt{2\pi}}
\int_\mathbb{R}
e^{-ix\xi}\phi(x)dx,
\quad
\xi\in\mathbb{R}.
\end{align*} 
We are ready to mention our main result.
\begin{theorem}\label{thm:main}
Let $2\le p<3$ and $\lambda\in\mathbb{C}$. 
Assume that $\mathrm{Im}\,\lambda>0$ and $(1+x^2)\varphi\in\Sigma$.
Let $T(\varepsilon)$ be the life span of (\ref{IVP}) 
defined by (\ref{def:lifespan}). 
Then we have 
\begin{align}\label{est:main}
\liminf_{\varepsilon\to 0}
\varepsilon^{2(p-1)/(3-p)}T(\varepsilon) 
\ge 
\left(
\frac{3-p}{2(p-1)(\mathrm{Im}\lambda)
\sup_{\xi\in\mathbb{R}}
|\widehat\varphi(\xi)|^{p-1}}
\right)^{2/(3-p)},
\end{align} 
where $\frac{1}{0}$ is understood as $+\infty$.
\end{theorem}
\begin{rem}
We see that the above estimate (\ref{est:main}) 
depends on $\mathrm{Im}\lambda$.
If $\varphi\in\Sigma$,  
then  
$\widehat\varphi$ is a bounded continuous function 
vanishing at infinity.
Therefore, 
$\sup_{\xi\in\mathbb{R}}
|\widehat\varphi(\xi)|^{p-1}$ 
is finite. 
\end{rem}
\begin{rem}
The estimates 
(\ref{est:O(delta)}),  
(\ref{est:w1-w2}), 
(\ref{est:Wexp(iG)}) and 
(\ref{est:Wpexp(iG)}) 
below
are essential to obtain 
main results. 
Unfortunately,  
such estimates can not be used 
in the case $1<p<2$. 
Therefore, 
in the case $1<p<2$ and $\mathrm{Im}\,\lambda>0$,
it is  still unknown whether 
(\ref{est:main}) holds, or not.
\end{rem}
In order to explain 
the estimate (\ref{est:main}) in detail, 
we introduce known results for 
the life span of classical solutions to 
the quasilinear Schr\"{o}dinger equation 
\begin{align}\label{cNLS}
\left\{
  \begin{array}{ll}
i\partial_t u +\frac{1}{2}\partial_x^2 u=
F(u,\partial_x u)
&\text{in $[0,\infty)\times\mathbb{R}$,}\\
u|_{t=0}=\varepsilon\phi       
&\text{on $\mathbb{R}$.}\\
  \end{array}
\right. 
\end{align}
Here, 
$\phi$ is sufficiently smooth and vanishes at infinity and 
$F$ is a gauge-invariant, cubic polynomial 
with respect to 
$u$, $\overline{u}$, $\partial_x u$ and $\overline{\partial_x u}$.
Let $S(\varepsilon)$ be the life span of the classical solution to 
(\ref{cNLS}). 
Then we see from Katayama--Tsutsumi \cite{KaTsu} that 
$
\liminf_{\varepsilon\to 0}\varepsilon^2\log S(\varepsilon) >0. 
$
Sunagawa \cite{Sunagawa} showed the following precise lower bound of 
$\liminf_{\varepsilon\to 0}\varepsilon^2\log S(\varepsilon)$:
\begin{align}\label{Sunagawa}
\liminf_{\varepsilon\to 0}
\varepsilon^2 \log S(\varepsilon) 
\ge
\dfrac{1}{
2\sup_{\xi\in\mathbb{R}}
|\widehat{\phi}(\xi)|^2
\mathrm{Im}F(1,i\xi)
}.
\end{align} 
From the estimate (\ref{Sunagawa}), 
we can expect some properties concerned with $S(\varepsilon)$. 
In particular,  
if either $\phi\equiv 0$ or 
\begin{align}\label{null}
\mathrm{Im}F(1,i\xi)\le 0,
\end{align}
then the right hand side of (\ref{Sunagawa}) is positive infinity  
and we hence expect that   
$S(\varepsilon)$ is much larger than $\exp(C/\varepsilon^2)$ 
for any $C>0$.  
In fact, 
Hayashi--Naumkin--Sunagawa \cite{HaNaSu}
recently proved 
the small data global existence
under the condition (\ref{null}). 
The estimate (\ref{Sunagawa}) is  
a (\ref{cNLS}) analogue of 
John and H\"{o}rmander's result   
concerned with quasilinear wave equations
(see \cite{John,Hormander}). 

Let us come back to the Cauchy problem (\ref{IVP}).
If $p=3$ and $\mathrm{Im}\lambda>0$, 
the result of \cite{Sunagawa} can be directly applicable 
to the (\ref{IVP}) cases. 
That is, it follows that 
\begin{align}\label{p=3}
\liminf_{\varepsilon\to 0}
\varepsilon^2 \log T(\varepsilon) 
\ge
\dfrac{1}{
2(\mathrm{Im}\lambda)
\sup_{\xi\in\mathbb{R}}
|\widehat{\varphi}(\xi)|^2
}.
\end{align} 
For any $2\le p<3$, 
the estimate (\ref{est:main}) 
can be understood as the (\ref{IVP}) version     
of (\ref{p=3}).
In fact, 
(\ref{p=3}) and (\ref{est:main}) are 
rewritten by the following form:
\begin{align*}
\liminf_{\varepsilon\to 0}
\int_1^{T(\varepsilon)}
\left(\frac{\varepsilon}{\sqrt{\tau}}\right)^{p-1}
d\tau
\ge 
\dfrac{1}{
(p-1)(\mathrm{Im}\lambda)
\sup_{\xi\in\mathbb{R}}
|\widehat{\varphi}(\xi)|^{p-1}
}.\\
\end{align*} 

We state our strategy for proving our main result.
The estimate (\ref{est:main}) 
formally follows from the method of 
\cite{Sunagawa}
(see also \cite{Hormander}, \cite{John}, etc.).
As the first step, 
we construct a suitable approximate solution $u_a(t,x)$ 
which is nearly equal to the modified free solution 
$m_p(t,x)=\varepsilon t^{-1/2}\exp(ix^2/2t)V(s_p(t),x/t)$,
where $V(s,\xi)$ solves the ordinary differential equation 
(\ref{ODE-1}) with $\phi_+=\widehat\varphi$. 
The function $m_p(t)$ is composed of the term  
$|\widehat{\varphi}|^{p-1}$ 
and 
the life span of $m_p(t)$ 
satisfies (\ref{est:main}).  
As the second step, 
we show an a priori estimate 
for the difference between $u(t)$ and $u_a(t)$, 
which enables us to see that 
$m_p(t)$ is close to $u(t)$ in some suitable sense.
However, 
in the sub-critical case $1<p<3$, 
some technical difficulty appears.  
In fact, 
although we have to treat higher order derivatives of $u_a(t)$ 
in the above second step,
the term $|\varphi|^{p-1}$
contained in $m_p(t,x)$
is not sufficiently smooth. 
In order to overcome such difficulty, 
we modify the first step. 
In more detail, 
we 
modify the original $u_a(t)$ 
by mollifying the power term  
$|\widehat\varphi|^{p-1}$. 
That is, we replace $|\widehat\varphi|^{p-1}$ by 
the mollified term 
$\rho_\delta\ast |\widehat\varphi|^{p-1}$,  
where 
$\ast$ is the convolution in $\mathbb{R}$ and 
$\rho_\delta$ ($\delta>0$) is some mollifier. 
Then we need to show that 
the modified $u_a$ is close to the original $u_a$ 
in some sense. 
For this purpose, we prove that 
the difference between 
$\rho_\delta\ast |\widehat\varphi|^{p-1}$ 
and 
$|\widehat\varphi|^{p-1}$ 
is estimated by 
\begin{align}\label{est:contents-1}
\| \rho_\delta\ast |\widehat\varphi|^{p-1} 
-
|\widehat\varphi|^{p-1}
\|_{H^1(\mathbb{R})}
\le \mathcal{O}(\delta), 
\end{align} 
where 
$\mathcal{O}$ is a non-negative increasing function 
tending to $0$ as $\delta\to 0$.
If we suitably take $\delta$ depending on $\varepsilon$, 
then   
we complete the modification of the above second step 
and hence the proof of (\ref{est:main}).

Listing the contents of this paper, 
we close this section. 
In Section \ref{sec:Preliminaries}, 
we state some preliminaries which will be useful  
to prove Theorem \ref{thm:main}.
In particular, the estimate (\ref{est:contents-1}) above is given.
In Section \ref{section:Approximate solution}, 
we next construct the modified $u_a(t,x)$ 
and prove some inequalities for the difference
between $u(t)$ and the modified $u_a(t)$.
In Section \ref{section:Proof of Theorem},
we finally show an a priori estimate 
which immediately implies 
Theorem \ref{thm:main}.
\section{Preliminaries}\label{sec:Preliminaries}
In this section, we show some preliminary properties
for proving Theorem \ref{thm:main}. 
For this purpose, we state some notation.
To consider derivatives of $|\widehat\varphi|^{p-1}$, 
we put a mollifier 
$\rho_\delta(x)=\delta^{-1}\rho(\delta^{-1}x)$  
for $\delta>0$. 
Here, $\rho$ is a smooth function 
on $\mathbb{R}$ satisfying 
$0\le\rho\le 1$,
$\mathrm{supp}\rho\subset (-1,1)$ 
and  
$\int_\mathbb{R}\rho(x)dx=1$. 
For $1\le q\le\infty$, 
we denote the $L^q(\mathbb{R})$-norm by $\|\cdot\|_q$. 
Recall 
the space $\Sigma$, 
the operator $J$ and  
the $X$-norm. 
For $T>0$, we define a set $X(T)$ by
\begin{align}\label{def:X(T)}
X(T)=
\left\{
w\in C([0,T);H^1({\mathbb{R}}));
Jw\in C([0,T);L^2({\mathbb{R}})), 
\sup_{t\in [0,T)}\| w(t)\|_{X}<\infty
\right\}.
\end{align} 
For multi-index 
$\alpha=(\alpha_1,\alpha_2)\in (\{ 0,1 \})^2$,
we set 
$Z^\alpha=\partial_x^{\alpha_1}J^{\alpha_2}$. 
Let $M(t)$ 
be a multiplication operator defined by 
\begin{align*}
M(t)=\exp{\left(\frac{x^2}{2t}\right)}.
\end{align*} 
Then we have the identity  
\begin{align}\label{identity:J}
J=M(t)(it\partial_x)M(-t).
\end{align}
The nonlinearity $\lambda |w|^{p-1}w$ 
is denoted by $\mathcal{N}(w)$.
For non-negative functions $f_1$ and $f_2$, 
we define $f_1\lesssim f_2$ 
if 
there exists some positive constant $C$ 
independent of 
$t$, $x$, $\varepsilon$ and $\delta$ 
such that 
$f_1\le Cf_2$.\\

The first proposition is proved by the standard argument  
(See, e.g., (2.5) in \cite{Sunagawa}). 
\begin{prop}\label{
prop:J-embedding}
For any $w\in X(T)$, we have 
\begin{align}
\| w(t)\|_\infty 
\lesssim (1+t)^{-1/2}\| w(t)\|_{X}, 
\quad
t\in [0,T).
\label{est:J}
\end{align}
\end{prop}

In Section \ref{section:Approximate solution} below, 
we deal with an approximate solution containing 
the mollified term 
$\rho_\delta\ast |\widehat\varphi|^{p-1}$.
Then the following estimate is essential to treat it:  
\begin{prop}\label{prop:mollifier}
Let $2\le p<3$ and $\varphi\in\Sigma$. 
There exists a non-negative, increasing function 
$\mathcal{O}$ on $(0,\infty)$ such that
\begin{align}\label{conv:O(delta)}
\mathcal{O}(\delta)\to 0 
\quad\text{as $\delta\to 0$} 
\end{align}  
and 
\begin{align}\label{est:O(delta)}
\left\| (\rho_\delta\ast |\widehat\varphi|^{p-1})
-
 |\widehat\varphi|^{p-1}\right\|_{H^1(\mathbb{R})} 
\le 
\mathcal{O}(\delta).
\end{align} 
\end{prop}
\begin{proof}
If $2\le p<3$, 
then the weak derivative of 
$|\widehat\varphi|^{p-1}$
is expressed by 
\begin{align}\label{identity:drivative-powerterm}
\partial_x |\widehat\varphi|^{p-1}=
\frac{p-1}{2}
|\widehat\varphi|^{p-3}
\mathrm{Re}
\left(
\widehat\varphi\ 
\overline{\partial_x \widehat\varphi}
\right) .
\end{align} 
Thus, we see that 
\begin{align*}
\| \partial_x |\widehat\varphi|^{p-1}\|_2 
\lesssim 
\| |\widehat\varphi|^{p-2}\|_\infty
\| \partial_x \widehat\varphi\|_2 
\lesssim 
\| \widehat\varphi\|_\infty^{p-2}
\| \partial_x \widehat\varphi\|_2 
\lesssim 
\| \varphi\|_\Sigma^{p-1},
\end{align*} 
where we have used  
the embedding 
$H^1(\mathbb{R})\hookrightarrow L^\infty(\mathbb{R})$,
the identity 
$\partial_x \widehat\varphi=-i\widehat{x\varphi}$ 
and 
the Plancherel theorem 
in the last inequality.
Hence it follows that 
$\partial_x |\widehat\varphi|^{p-1}\in L^2(\mathbb{R})$ 
and 
\begin{align*}
\| \rho_\delta\ast |\widehat\varphi|^{p-1}
-
|\widehat\varphi|^{p-1}\|_{H^1(\mathbb{R})} \to 0
\quad\text{as $\delta\to 0$.}
\end{align*} 
If we put for any $\delta$,
\begin{align*}
\mathcal{O}(\delta)
=
\min\left\{ 
2\| |\widehat\varphi|^{p-1}\|_{H^1(\mathbb{R})},
\sup_{0<\eta\le\delta}
\| \rho_\eta\ast |\widehat\varphi|^{p-1}
-
|\widehat\varphi|^{p-1}\|_{H^1(\mathbb{R})}
\right\},
\end{align*} 
then $\mathcal{O}$ 
is a non-negative, increasing function satisfying 
(\ref{conv:O(delta)}) and (\ref{est:O(delta)}).
\end{proof}
 
In Sections \ref{section:Approximate solution} and 
\ref{section:Proof of Theorem} below, 
we treat the $X$-norm of 
the difference of two nonlinearities
$\mathcal{N}(w_1) - \mathcal{N}(w_2)$.
Then the following estimate is useful:
\begin{prop}\label{prop:w1-w2}
Let $2\le p<3$ and $\lambda\in\mathbb{C}$.
Suppose that $w_j \in X(T)$, $j=1,2$ and $T>0$.
Then we have
\begin{align}\label{est:w1-w2}
\| &\mathcal{N}(w_1(t))-
\mathcal{N}(w_2(t))\|_{X}\nonumber\\
&\lesssim
(1+t)^{-(p-1)/2}
\sup_{j=1,2}\| w_j(t)\|_{X}^{p-1}
\| w_1(t)-w_2(t)\|_{X},
\quad\text{$0\le t<T$.}
\end{align}
\end{prop}
\begin{proof}
By a direct calculation, 
we obtain 
\begin{align*}
\partial_x &\mathcal{N}(w_1) 
-
\partial_x\mathcal{N}(w_2) \\
&=
\frac{p+1}{2}
\left(
|w_1|^{p-1}-|w_2|^{p-1}
\right)
\partial_x w_1 
+
\frac{p+1}{2}
|w_2|^{p-1}
\left(
\partial_x w_1-
\partial_x w_2
\right) \\
&\quad +
\frac{p-1}{2}
\left(
|w_1|^{p-3}w_1^2 -|w_2|^{p-3}w_2^2
\right)
\overline{\partial_x w_1} 
+
\frac{p-1}{2}
|w_2|^{p-3}w_2^2
\overline{\left(
\partial_x w_1-
\partial_x w_2
\right)}.
\end{align*} 
Using the identity (\ref{identity:J}),
it follows that 
\begin{align*}
J &\mathcal{N}(w_1) 
-
J \mathcal{N}(w_2) \\
&=
M(t)(it)\partial_x
\left(
\mathcal{N}(M(-t)w_1(t))
-
\mathcal{N}(M(-t)w_2(t))
\right)\\
&=
\frac{p+1}{2}
\left(
|w_1|^{p-1}-|w_2|^{p-1}
\right)
M(t)(it)\partial_x M(-t) w_1(t)\\
&\quad +
\frac{p+1}{2}
|w_2|^{p-1}
M(t)(it)\partial_x M(-t)
(w_1-w_2) \\
&\quad +
\frac{p-1}{2}
M(t)(it)
\left(
|w_1|^{p-3}(M(-t)w_1)^2 
- 
|w_2|^{p-3}(M(-t)w_2)^2
\right)
\overline{\partial_x M(-t)w_1}\\
&\quad +
\frac{p-1}{2}
M(t)(it)
|w_2|^{p-3}(M(-t)w_2)^2
\overline{
\partial_x M(-t)(w_1-w_2)}\\
&=
\frac{p+1}{2}
\left(
|w_1|^{p-1}-|w_2|^{p-1}
\right)
J w_1 
+
\frac{p+1}{2}
|w_2|^{p-1}
\left(
J w_1-
J w_2
\right) \\
&\quad -
\frac{p-1}{2}
\left(
|w_1|^{p-3}w_1^2 -|w_2|^{p-3}w_2^2
\right)
\overline{J w_1} 
-
\frac{p-1}{2}
|w_2|^{p-3}w_2^2
\overline{\left(
J w_1-
J w_2
\right)}.
\end{align*}  
Therefore, we obtain 
\begin{align*}
\| & \mathcal{N}(w_1) 
-
 \mathcal{N}(w_2) 
\|_{X}\\
&\lesssim 
\| |w_1(t)|^{p-1}-|w_2(t)|^{p-1}\|_\infty 
\| w_1(t)\|_{X}
+
\| w_2\|_\infty^{p-1}
\| w_1(t)-w_2(t)\|_{X}\\
&\quad +
\| |w_1(t)|^{p-3}w_1(t)^2
-
|w_2(t)|^{p-3}w_2(t)^2\|_\infty 
\| w_1(t)\|_{X}.
\end{align*} 
To complete the proof of the proposition, 
we mention the following lemma:
\begin{lem}\label{lem:appendix}
Let $q\ge 2$. 
For any $\alpha_1,\alpha_2\in\mathbb{C}$, 
we have 
\begin{align}
\left| |\alpha_1|^{q-1} 
-
|\alpha_2|^{q-1} \right|
&\lesssim 
|\alpha_1-\alpha_2| 
\left(
|\alpha_1|+|\alpha_2|
\right)^{q-2},\label{est:appendix-1}\\
\left| 
|\alpha_1|^{q-3}\alpha_1^2 
-
|\alpha_2|^{q-3}\alpha_2^2 \right|
&\lesssim 
|\alpha_1-\alpha_2| 
\left(
|\alpha_1|+|\alpha_2|
\right)^{q-2}. \label{est:appendix-2}
\end{align} 
\end{lem}
From Lemma \ref{lem:appendix} 
and (\ref{est:J}), 
we have (\ref{est:w1-w2}).
\end{proof}
We now prove Lemma \ref{lem:appendix} above.
Let $q\ge 2$ and $j=1,2$. 
We set $\alpha_j=r_j e^{i\theta_j}$, 
where $r_j>0$ and $\theta_j\in (-\pi,\pi]$.
It follows from the mean value theorem that
\begin{align}\label{est:appendix-1-1}
|r_1^{q-1}-r_2^{q-1}|
\lesssim 
|r_1-r_2|(r_1+r_2)^{q-2}.
\end{align} 
Since 
\begin{align}\label{est:appendix-1-2}
|r_1-r_2|\le |\alpha_1-\alpha_2|, 
\end{align}
we obtain (\ref{est:appendix-1}).
On the other hand, we see that    
\begin{align*}
| &r_1^{q-1}e^{2i\theta_1} 
-
r_2^{q-1}e^{2i\theta_2} | \\ 
&\lesssim
|r_1^{q-1}e^{2i\theta_1}-r_2^{q-1}e^{2i\theta_1}|
+
r_2^{q-1}
|e^{2i\theta_1}-e^{2i\theta_2}| \\
&\lesssim
|r_1^{q-1}-r_2^{q-1}|
+
r_2^{q-1}
|e^{i\theta_1}-e^{i\theta_2}|
|e^{i\theta_1}+e^{i\theta_2}| \\
&\lesssim
|r_1^{q-1}-r_2^{q-1}|
+
r_2^{q-2}
|r_1 e^{i\theta_1}-r_2 e^{i\theta_2}|
+
r_2^{q-2}
|r_1 e^{i\theta_1}-r_2 e^{i\theta_1}| \\
&\lesssim
|r_1^{q-1}-r_2^{q-1}|
+
r_2^{q-2}
|r_1 e^{i\theta_1}-r_2 e^{i\theta_2}|
+
r_2^{q-2}
|r_1 -r_2| .
\end{align*}
From (\ref{est:appendix-1-1}) and (\ref{est:appendix-1-2}), 
we have (\ref{est:appendix-2}). 
\section{Approximate solution
}\label{section:Approximate solution}
In this section, 
we suppose that $2\le p<3$ and $\mathrm{Im}\lambda >0$, 
we define an approximate solution $u_a(t,x)$ 
and show some estimates  
dividing five subsections.
Our goal of this subsection is to prove 
the following two inequalities 
which are important to show Theorem \ref{thm:main}:  
\begin{align}\label{est:u-a}
\| u_a(t)\|_{X}\lesssim \varepsilon, 
\quad 0<t<T_B(\varepsilon)
\end{align}
and
\begin{align}\label{est:int-R}
\int_0^{T_B(\varepsilon)}\| R(t)\|_{X} dt 
\lesssim
\varepsilon^{3/2}
+
\varepsilon\mathcal{O}(\varepsilon^{1/4}),
\end{align}
where  
\begin{align}\label{def:T-B}
T_B(\varepsilon)
=
\left(\frac{(3-p)B}{2\varepsilon^{p-1}}
\right)^{2/(3-p)},
\end{align} 
$B$ is some positive number, 
$R=\mathcal{L}u_a - \mathcal{N}(u_a)$ 
and 
$\mathcal{L}=i\partial_t +\frac{1}{2}\partial_x^2$.
Inequalities (\ref{est:u-a}) and (\ref{est:int-R})
are shown in Subsections 
\ref{step4} and \ref{step5}, 
respectively.

\subsection{Definition of $V(t,x)$}\label{step1}
Assume that $(1+x^2)\varphi\in\Sigma$. 
We consider 
an ordinary differential equation 
\begin{align}\label{ODE}
\left\{
  \begin{array}{ll}
i\partial_s V(s,\xi)
=
\mathcal{N}(V(s,\xi)),&
(s,\xi)\in [0,B]\times\mathbb{R},\\
V(0,\xi)=e^{-i\pi/4}\widehat\varphi(\xi),&
\xi\in\mathbb{R}\\
  \end{array}
\right.
\end{align} 
for some $B>0$. 
We define $A\in (0,\infty]$ by
\begin{align}\label{def:A}
A^{-1}
=
(p-1)(\mathrm{Im}\lambda)
\sup_{\xi\in\mathbb{R}}
|\widehat\varphi(\xi)|^{p-1}.
\end{align} 
Then the solution to (\ref{ODE}) with $B\in (0,A)$ 
is expressed by 
\begin{align*}
V(s,\xi)
=
W(s,\xi)^{-1/(p-1)}\exp(iG(s,\xi))\widehat\varphi(\xi).
\end{align*} 
Here, 
\begin{align*}
W(s,\xi)
&=
1-(p-1)\mathrm{Im}\lambda|\widehat\varphi(\xi)|^{p-1}s,\\
G(s,\xi)
&=
-\mathrm{Re}\lambda
|\widehat\varphi(\xi)|^{p-1}
\int_0^s W(\sigma,\xi)^{-1}d\sigma
-\frac{\pi}{4}.
\end{align*}
In order to prove (\ref{est:int-R}), 
we need to estimate $\partial_{\xi}^3 V(s,\xi)$
(see Subsection \ref{step2}). 
However, $V(s,\cdot)$ generally 
does not belong to $C^2(\mathbb{R})$. 
Therefore, we have to mollify $V$.
For $\delta>0$, we define $V_\delta(s,\xi)$ by
\begin{align*}
V_\delta(s,\xi)
=
W_\delta(s,\xi)^{-1/(p-1)}
\exp(iG_\delta(s,\xi))
\widehat\varphi(\xi).
\end{align*} 
Here, 
\begin{align*}
W_\delta(s,\xi)
&=
1-(p-1)\mathrm{Im}\lambda
(\rho_\delta\ast|\widehat\varphi|^{p-1})(\xi)s,\\
G_\delta(s,\xi)
&=
-\mathrm{Re}\lambda
(\rho_\delta\ast|\widehat\varphi|^{p-1})(\xi)
\int_0^s W_\delta(\sigma,\xi)^{-1}d\sigma
-\frac{\pi}{4}.
\end{align*}
Then we see the following property of $V_\delta$:
\begin{prop}\label{prop:V}
Let $2\le p<3$, $\lambda\in\mathbb{C}$, 
$\delta>0$ and $B\in (0,A)$. 
Assume that $\mathrm{Im}\,\lambda>0$ and $(1+x^2)\varphi\in\Sigma$.
Then we have 
$W_\delta^{-1/(p-1)}\exp(iG_\delta)
\in C^\infty([0,B]\times\mathbb{R})$. 
Furthermore, it follows that 
for $l=0,1$ and $m=0,1,2,3$, 
\begin{align}\label{est:Wexp(iG)}
\sup_{(s,\xi)\in [0,B]\times\mathbb{R}}
\left|
\partial_s^l \partial_{\xi}^{m} 
\left(
W_\delta^{-1/(p-1)}(s,\xi)
\exp(iG_\delta(s,\xi))
\right)
\right|
\lesssim \delta^{\min\{0,1-m\}}
\end{align}
and that 
\begin{align}\label{identity:V}
i&\partial_s V_\delta(s,\xi)
-
\mathcal{N}(V_\delta(s,\xi))\nonumber\\
&=
\lambda W_\delta^{-p/(p-1)}(s,\xi)
\exp(iG_\delta(s,\xi))
\widehat\varphi(\xi)
\left(\rho_\delta\ast|\widehat\varphi|^{p-1}(\xi)-
|\widehat\varphi|^{p-1}(\xi)\right).
\end{align}
\end{prop}
\begin{proof}
For $\mathrm{Im}\lambda>0$ and 
$(s,\xi)\in [0,B]\times\mathbb{R}$,
we see from the H\"{o}lder-Young inequality that  
\begin{align*}
W_\delta(s,\xi)
&\ge 
1-(p-1)\sup_{\xi\in\mathbb{R}}\mathrm{Im}\lambda
(\rho_\delta\ast|\widehat\varphi|^{p-1})(\xi)s\\
&\ge 
1-(p-1)\sup_{\xi\in\mathbb{R}}\mathrm{Im}\lambda
|\widehat\varphi|^{p-1}(\xi)s\\
&\ge 
1-B/A.
\end{align*}
Therefore, we obtain 
\begin{align}\label{est:1/W}
0\le W_\delta^{-1}(s,\xi) \le \frac{A}{A-B}.
\end{align}
Thus,  
$V_\delta$ is well-defined on 
$[0,B]\times\mathbb{R}$.

Using the H\"{o}lder-Young inequality and 
(\ref{identity:drivative-powerterm}),
we have for $m=1,2,\cdots$, 
\begin{align*}
\| 
\partial_x^m (\rho_\delta\ast|\widehat\varphi|^{p-1})
\|_\infty
&=
\| 
\partial_x^{m-1}\rho_\delta \ast
\partial_x|\widehat\varphi|^{p-1} 
\|_\infty\\
&\lesssim 
\delta^{1-m}
\| 
\partial_x|\widehat\varphi|^{p-1} 
\|_\infty \\
&\lesssim 
\delta^{1-m}
\| 
\widehat\varphi
\|_\infty^{p-2} 
\| 
\partial_x \widehat\varphi
\|_\infty \\
&\lesssim 
\delta^{1-m}
\| 
(1+x^2)\varphi
\|_\Sigma. 
\end{align*}  
Then it follows
from the definition of $W_\delta$ and $G_\delta$ that 
$W_\delta^{-1/(p-1)}\exp(iG)\in 
C^\infty([0,B]\times\mathbb{R})$
and that (\ref{est:Wexp(iG)}) holds.
The identity (\ref{identity:V}) 
is given by a direct calculation.
\end{proof}
\begin{rem}
From the proof of the above proposition, 
we immediately see that  
\begin{align}\label{est:Wpexp(iG)}
\sup_{(s,\xi)\in [0,B]\times\mathbb{R}}
\left| 
\partial_{\xi}^{m} 
\left(
W^{-p/(p-1)}_\delta(s,\xi)\exp(iG(s,\xi))
\right)
\right|
\lesssim 1,
\end{align} 
for  
$m=0,1$,
which is applied later.
\end{rem}
\subsection{Definition of $m(t,x)$ and $Q(t,x)$
}\label{step2}
Assume that $\delta>0$, $B\in (0,A)$ and  
$(t,x)\in (1,T_B(\varepsilon)]\times\mathbb{R}$.
Let $m=m(t,x)$ be a function defined by
\begin{align*}
m(t,x)
=
\frac{\varepsilon M(t)}{t^{1/2}}V_\delta(s(t),\xi(t,x)),
\end{align*} 
where
\begin{align*}
s(t)
=
\int_0^t
\left(
\frac{\varepsilon}{\tau^{1/2}}
\right)^{p-1}
d\tau
=
\frac{2\varepsilon^{p-1}t^{(3-p)/2}}{3-p} 
\quad\text{and}\quad
\xi(t,x)
=
\frac{x}{t}.
\end{align*} 
Furthermore, we define a function $Q=Q(t,x)$ by 
\begin{align*}
Q(t,x)
=
\mathcal{L}m(t,x)
-
\mathcal{N}(m(t,x)).
\end{align*} 
We see from (\ref{identity:V}) that $Q$ is expressed by 
\begin{align}\label{identity:Q}
Q(t,x)
&=
i\frac{\varepsilon M(t)}{t^{1/2}}
\left\{
-\frac{ix^2}{2t^2}-
\frac{1}{2t}-
\frac{x}{t^2}\partial_{\xi}+
\varepsilon^{p-1}t^{-(p-1)/2}\partial_s
\right\}
V_\delta(s(t),\xi(t,x))\nonumber\\
&\quad +
\frac{1}{2}\frac{\varepsilon M(t)}{t^{1/2}}
\left\{
\frac{-x^2}{t^2}+
\frac{i}{t}+
\frac{2ix}{t^2}\partial_{\xi}+
\frac{1}{t^2}\partial_{\xi}^2
\right\}
V_\delta(s(t),\xi(t,x))\nonumber\\
&=
\frac{\varepsilon^pM(t)}{t^{p/2}}
\left\{
i\partial_s V_\delta(s(t),\xi(t,x))
\right\} +
\frac{\varepsilon M(t)}{2t^{5/2}}
\partial_{\xi}^2 V(s(t),\xi(t,x))\nonumber\\
&=
\frac{\lambda\varepsilon^pM(t)}{t^{p/2}}
\left(
W_\delta^{-p/(p-1)}\exp(iG_\delta)
\widehat\varphi
(\rho_\delta\ast|\widehat\varphi|^{p-1}-
|\widehat\varphi|^{p-1})
\right) (s(t),\xi(t,x))\nonumber\\
&\quad +
\frac{\varepsilon M(t)}{2t^{5/2}}
\partial_{\xi}^2 V_\delta(s(t),\xi(t,x))\\
&=:Q_1(t,x)+Q_2(t,x)\nonumber.
\end{align}
\subsection{Estimates of $m(t,x)$ and $Q(t,x)$
}\label{step3}
In this subsection, 
we assume that 
$\delta>0$, $B\in (0,A)$ and 
$(t,x)\in (1,T_B(\varepsilon)]\times\mathbb{R}$, 
and we estimate 
the $X$-norm of $m(t)$ and $Q(t)$. 
The first derivative of $m$ is given by 
\begin{align*}
\partial_x m(t,x)
&=
\partial_x\left\{
\frac{\varepsilon M(t)}{t^{1/2}}
\left(\left(W^{-1/(p-1)}\exp(iG_\delta)\right)
\left(s(t),\left(\frac{x}{t}\right)\right)\right)
\left(\widehat\varphi\left(\frac{x}{t}\right)\right)
\right\}\\
&=
\frac{\varepsilon M(t)}{t^{1/2}}
\left(\left(W^{-1/(p-1)}\exp(iG_\delta)\right)
\left(s(t),\left(\frac{x}{t}\right)\right)\right)
\left(\frac{ix}{t}
\widehat\varphi\left(\frac{x}{t}\right)\right)\\
&\quad +
\frac{\varepsilon M(t)}{t^{1/2}}
\left(\frac{1}{t}\partial_{\xi}
\left(W^{-1/(p-1)}\exp(iG_\delta)\right)
\left(s(t),\left(\frac{x}{t}\right)\right)\right)
\left(\widehat\varphi\left(\frac{x}{t}\right)\right)\\
&\quad +
\frac{\varepsilon M(t)}{t^{1/2}}
\left(
\left(W^{-1/(p-1)}\exp(iG_\delta)\right)
\left(s(t),\left(\frac{x}{t}\right)\right)\right)
\left(\frac{1}{t}\partial_{\xi}
\widehat\varphi\left(\frac{x}{t}\right)\right) .
\end{align*}
The identity (\ref{identity:J}) implies that 
\begin{align*}
J m(t,x)
&=
\frac{i\varepsilon M(t)}{t^{1/2}}
\left(\partial_{\xi}
\left(W^{-1/(p-1)}\exp(iG_\delta)\right)
\left(s(t),\left(\frac{x}{t}\right)\right)\right)
\left(\widehat\varphi\left(\frac{x}{t}\right)\right)\\
&\quad +
\frac{i\varepsilon M(t)}{t^{1/2}}
\left(
\left(W^{-1/(p-1)}\exp(iG_\delta)\right)
\left(s(t),\left(\frac{x}{t}\right)\right)\right)
\left(\partial_{\xi}
\widehat\varphi\left(\frac{x}{t}\right)\right) .
\end{align*}  
From (\ref{est:Wexp(iG)}), we obtain 
\begin{align}\label{est:m}
\| m(t)\|_{X} 
&\lesssim
\frac{\varepsilon}{t^{1/2}}
\left(
\left\|
\widehat\varphi\left(\frac{\cdot}{t}\right)
\right\|_2
+
\left\|
\frac{\cdot}{t}
\widehat\varphi\left(\frac{\cdot}{t}\right)
\right\|_2
+
\left\|
\partial_\xi 
\widehat\varphi\left(\frac{\cdot}{t}\right)
\right\|_2
\right) \nonumber\\
&\lesssim 
\varepsilon
\|
\varphi
\|_\Sigma 
\lesssim 
\varepsilon.
\end{align} 
We next estimate $Q_1(t,x)$. 
It follows that 
\begin{align*}
&\partial_x Q_1(t,x)\\
&=
\partial_x\left\{
\frac{\lambda\varepsilon^p M(t)}{t^{p/2}}
\left(\left(W^{-p/(p-1)}\exp(iG_\delta)\right)
\left(s(t),\left(\frac{x}{t}\right)\right)\right)
\left(\left( 
\widehat\varphi
\left(
\rho_\delta\ast|\widehat\varphi|^{p-1}
-
|\widehat\varphi|^{p-1}
\right)\right)
\left(\frac{x}{t}\right)\right)
\right\}\\
&=
\frac{\lambda\varepsilon^p M(t)}{t^{p/2}}\\
&\ \times 
\biggl\{
\left(\left
(W^{-p/(p-1)}\exp(iG_\delta)\right)
\left(s(t),\left(\frac{x}{t}\right)\right)\right)
\left(
\left(
\rho_\delta\ast|\widehat\varphi|^{p-1}
-
|\widehat\varphi|^{p-1}
\right)
\left(\frac{x}{t}\right)\right)
\left( 
\frac{ix}{t}\widehat\varphi
\left(\frac{x}{t}\right)
\right)\\
&\quad +
\left(\frac{1}{t}\partial_{\xi}
\left(W^{-p/(p-1)}\exp(iG_\delta)\right)
\left(s(t),\left(\frac{x}{t}\right)\right)\right)
\left(
\left(
\rho_\delta\ast|\widehat\varphi|^{p-1}
-
|\widehat\varphi|^{p-1}
\right)
\left(\frac{x}{t}\right)\right)
\left( 
\widehat\varphi
\left(\frac{x}{t}\right)
\right)\\
&\quad +
\left(\left(W^{-p/(p-1)}\exp(iG_\delta)\right)
\left(s(t),\left(\frac{x}{t}\right)\right)\right)
\left(\frac{1}{t}\partial_{\xi}
\left(
\rho_\delta\ast|\widehat\varphi|^{p-1}
-
|\widehat\varphi|^{p-1}
\right)
\left(\frac{x}{t}\right)\right)
\left( 
\widehat\varphi
\left(\frac{x}{t}\right)
\right)\\
&\quad +
\left(\left(W^{-p/(p-1)}\exp(iG_\delta)\right)
\left(s(t),\left(\frac{x}{t}\right)\right)\right)
\left(
\left(
\rho_\delta\ast|\widehat\varphi|^{p-1}
-
|\widehat\varphi|^{p-1}
\right)
\left(\frac{x}{t}\right)\right)
\left(\frac{1}{t}\partial_{\xi}
\widehat\varphi
\left(\frac{x}{t}\right)
\right)\biggr\}
\end{align*}  
and that
\begin{align*}
&J Q_1(t,x)\\
&=
\frac{\varepsilon^p M(t)}{t^{p/2}}\\
&\ \times 
\biggl\{
\left(\partial_{\xi}
\left(W^{-p/(p-1)}\exp(iG_\delta)\right)
\left(s(t),\left(\frac{x}{t}\right)\right)\right)
\left(
\left(
\rho_\delta\ast|\widehat\varphi|^{p-1}
-
|\widehat\varphi|^{p-1}
\right)
\left(\frac{x}{t}\right)\right)
\left( 
\widehat\varphi
\left(\frac{x}{t}\right)
\right)\\
&\quad +
\left(\left(W^{-p/(p-1)}\exp(iG_\delta)\right)
\left(s(t),\left(\frac{x}{t}\right)\right)\right)
\left(\partial_{\xi}
\left(
\rho_\delta\ast|\widehat\varphi|^{p-1}
-
|\widehat\varphi|^{p-1}
\right)
\left(\frac{x}{t}\right)\right)
\left( 
\widehat\varphi
\left(\frac{x}{t}\right)
\right)\\
&\quad +
\left(\left(W^{-p/(p-1)}\exp(iG_\delta)\right)
\left(s(t),\left(\frac{x}{t}\right)\right)\right)
\left(
\left(
\rho_\delta\ast|\widehat\varphi|^{p-1}
-
|\widehat\varphi|^{p-1}
\right)
\left(\frac{x}{t}\right)\right)
\left(
\widehat\varphi
\left(\frac{x}{t}\right)
\right)\biggr\}.
\end{align*} 
We see from (\ref{est:Wpexp(iG)}) 
and (\ref{est:O(delta)}) that 
\begin{align*}
\| Q_1(t)\|_{X}
&\lesssim
\varepsilon^p
t^{-\frac{p}{2}}
\mathcal{O}(\delta)
\left(
\left\|
\widehat\varphi\left(\frac{\cdot}{t}\right)
\right\|_\infty
+
\left\|
\frac{\cdot}{t}
\widehat\varphi\left(\frac{\cdot}{t}\right)
\right\|_\infty
+
\left\|
\partial_\xi 
\widehat\varphi\left(\frac{\cdot}{t}\right)
\right\|_\infty
\right) \\
&\lesssim
\varepsilon^p
t^{-(p-1)/2}
\mathcal{O}(\delta)
\| (1+x^2)\varphi\|_\Sigma.
\end{align*} 
By (\ref{est:Wexp(iG)}), 
the remainder $Q_2(t,x)$ is estimated by 
\begin{align*}
\| Q_2(t)\|_{X}
\lesssim
\varepsilon
t^{-2}\delta^{-2}.
\end{align*} 
Hence we obtain
\begin{align}\label{est:Q}
\| Q(t)\|_{X}
\lesssim 
\varepsilon^p
t^{-(p-1)/2}
\mathcal{O}(\delta)
+
\varepsilon
t^{-2}\delta^{-2}.
\end{align} 
\subsection{Definition of $u_a(t,x)$ and $R(t,x)$
}\label{step4}
Assume that 
$\delta>0$ and  
$B\in (0,A)$.  
Let $\chi$ be a smooth function on $\mathbb{R}$ 
satisfying 
$0\le \chi \le 1$, 
$\chi(t)=1$ if $t\le 1$
and 
$\chi(t)=0$ if $t\ge 2$.
For $\varepsilon>0$ 
and  
$(t,x)\in (0,T_B(\varepsilon)]\times\mathbb{R}$, 
we put 
\begin{align*}
u_a(t,x)
=
\chi(\varepsilon t)U(t)(\varepsilon\varphi(x))
+
(1-\chi(\varepsilon t))m(t,x),
\end{align*}
where $U(t)$ is the free Schr\"{o}dinger propagator.
That is, 
$u_{0,\varepsilon}(t)=U(t)(\varepsilon\varphi)$ solves 
\begin{align*}
\left\{
  \begin{array}{ll}
i\partial_t u_{0,\varepsilon} +
\frac{1}{2}\partial_x^2 u_{0,\varepsilon}=0
&\text{in $[0,\infty)\times\mathbb{R}$,}\\
u_{0,\varepsilon}|_{t=0}=\varepsilon\varphi       
&\text{on $\mathbb{R}$.}\\
  \end{array}
\right. 
\end{align*}
From (\ref{est:m}) and the standard equality
\begin{align}\label{equality:free}
\| U(t)(\varepsilon\varphi)\|_{X}=
\varepsilon\| \varphi\|_\Sigma,
\end{align} 
we have (\ref{est:u-a}). 

Let $R=R(t,x)$ be a function defined by 
\begin{align*}
R(t,x)
=
\mathcal{L}u_a(t,x)
-
\mathcal{N}(u_a(t,x)), 
\quad 
(t,x)\in (0,T_B(\varepsilon)]\times\mathbb{R}.
\end{align*} 
\subsection{Estimates of $R(t,x)$}\label{step5}
As the final step, we estimate $R$.
Henceforth, we fix $\delta=\varepsilon^{1/4}$.
If $t\in (0,1/\varepsilon]$, 
then 
$R(t,x)=-\chi(\varepsilon t)
\mathcal{N}(U(t)(\varepsilon\varphi(x)))$. 
Therefore, it follows from 
(\ref{est:w1-w2}) and (\ref{equality:free}) that 
\begin{align}\label{est:R-1}
\| R(t)\|_{X}
\lesssim
\varepsilon^p t^{-(p-1)/2},
\quad
t\in (0,1/\varepsilon].
\end{align}
Thus, we have 
\begin{align}\label{est:int-R-1}
\int_{0}^{1/\varepsilon}
\| R(t)\|_{X}dt
\lesssim
\varepsilon^{3/2},
\end{align}
where we have used the condition $p\ge 2$.

In the case $t\in [2/\varepsilon,T_B(\varepsilon)]$, 
$R(t,x)$ is equal to $Q(t,x)$. 
We hence obtain 
\begin{align*}
\| R(t)\|_{X}
\lesssim 
\varepsilon^p
t^{-(p-1)/2}
\mathcal{O}(\delta)
+
\varepsilon
t^{-2}\delta^{-2},
\quad
t\in [2/\varepsilon,T_B(\varepsilon))
\end{align*}
and 
\begin{align}\label{est:int-R-3}
\int_{2/\varepsilon}^{T_B(\varepsilon)}
\| R(t)\|_{X}dt
\lesssim
\varepsilon\mathcal{O}(\varepsilon^{1/4})
+
\varepsilon^{3/2}.
\end{align}

Let us consider the case 
$t\in (1/\varepsilon,2/\varepsilon)$. 
Then we have 
\begin{align}\label{identity:R-2}
R(t,x)
&=
i\varepsilon\chi^\prime(\varepsilon t)
(u_{0,\varepsilon}(t,x)-m(t,x))\nonumber\\
&\quad + 
(1-\chi(\varepsilon t))
(\mathcal{N}(m(t,x))-\mathcal{N}(u_a(t,x)))\nonumber\\
&\quad -
\chi(\varepsilon t)\mathcal{N}(u_a(t,x))\nonumber\\
&\quad +
(1-\chi(\varepsilon t))Q(t,x).
\end{align}
It is well-known that 
the free solution $u_{0,\varepsilon}$ 
is expressed by 
\begin{align*}
U(t)(\varepsilon\varphi(x))
=
\frac{\varepsilon M(t)}{\sqrt{2\pi it}}
\int_{-\infty}^\infty
\exp\left( -\frac{ixy}{t}\right)
\exp\left( \frac{iy^2}{2t}\right)
\varphi(y)dy.
\end{align*} 
Therefore, we obtain for any 
$t\in (1/\varepsilon,2/\varepsilon)$, 
\begin{align*}
U(t)(\varepsilon\varphi(x))-m(t,x)
&=
\frac{\varepsilon M(t)}{t^{1/2}}
\left\{
V_\delta\left( 0,\frac{x}{t} \right)
-
V_\delta\left( s(t),\frac{x}{t} \right)
\right\}\\
&\quad +
\frac{\varepsilon M(t)}{\sqrt{2\pi it}}
\int_{-\infty}^\infty
\exp\left( -\frac{ixy}{t}\right)
\left\{ \exp\left( \frac{iy^2}{2t}\right)-1 \right\}
\varphi(y)dy \\
&=
\frac{2\varepsilon M(t)}{(3-p)t^{1/2}}
t^{(3-p)/2}
\varepsilon^{p-1}
\left\{
-\int_0^1
\partial_s
V_\delta\left( s(t)\theta,\frac{x}{t} \right)
d\theta
\right\}\\
&\quad +
\frac{\varepsilon M(t)}{\sqrt{2\pi it}}
\int_{-\infty}^\infty
\exp\left( -\frac{ixy}{t}\right)
\left\{ \exp\left( \frac{iy^2}{2t}\right)-1 \right\}
\varphi(y)dy \\
&=:
f_1(t,x)+f_2(t,x).
\end{align*} 
We see from (\ref{est:Wexp(iG)}) that 
\begin{align*}
\| f_1(t)\|_{X}
\lesssim
\varepsilon^p t^{(3-p)/2},
\quad 
t\in (1/\varepsilon,2/\varepsilon).
\end{align*} 
For the other term $f_2$,  
the following estimate was shown
by \cite{Sunagawa}:
\begin{align*}
\| f_2(t)\|_{X}
\lesssim
\varepsilon t^{-1},
\quad 
t>0.
\end{align*} 
Thus, we obtain for any 
$t\in (1/\varepsilon,2/\varepsilon)$, 
\begin{align}\label{est:R-2-1}
\| U(t)(\varepsilon\varphi)-m(t)\|_{X}
\lesssim
\varepsilon^{3/2},
\end{align}
where we have used the relation 
$1/\varepsilon \le t\le 2/\varepsilon$ 
and $p\ge 2$.
By (\ref{est:w1-w2}) and (\ref{est:u-a}), 
we have
\begin{align*}
\| \mathcal{N}(m(t))-\mathcal{N}(u_a(t)) \|_{X}
\lesssim
\varepsilon^{p-1}
t^{-(p-1)/2}
\| m(t)-u_a(t)\|_{X},
\quad 
t\in (1/\varepsilon,2/\varepsilon).
\end{align*} 
Since $m-u_a=\chi(\varepsilon t)(m-u_{0,\varepsilon})$,
it follows from (\ref{est:R-2-1}) that 
\begin{align}\label{est:R-2-2}
\| \mathcal{N}(m(t))-\mathcal{N}(u_a(t)) \|_{X}
\lesssim
\varepsilon^3 ,
\quad 
t\in (1/\varepsilon,2/\varepsilon),
\end{align}
where we have used the relation 
$1/\varepsilon \le t\le 2/\varepsilon$ 
and $p\ge 2$ again.
By the same argument as in the proof of (\ref{est:R-1}), 
we see that 
\begin{align}\label{est:R-2-3}
\| \mathcal{N}(u_a(t,x)) \|_{X}
\lesssim 
\varepsilon^{5/2},
\quad 
t\in (1/\varepsilon,2/\varepsilon).
\end{align} 
Therefore, it follows from 
(\ref{identity:R-2})--(\ref{est:R-2-3}) and 
(\ref{est:Q})
that 
\begin{align}\label{est:int-R-2}
\int_{1/\varepsilon}^{2/\varepsilon}
\| R(t)\|_{X} dt 
\lesssim 
\varepsilon^{3/2}(1+\mathcal{O}(\varepsilon^{1/4})).
\end{align} 

We are ready to show (\ref{est:int-R}).
The estimates (\ref{est:int-R-1}), 
(\ref{est:int-R-2}) and (\ref{est:int-R-3}) 
enable us to see that 
\begin{align*}
\int_0^{T_B(\varepsilon)}
\| R(t)\|_{X} dt 
&\le
\left(
\int_0^{1/\varepsilon}
+
\int_{1/\varepsilon}^{2/\varepsilon}
+
\int_{2/\varepsilon}^{T_B(\varepsilon)}
\right)
\| R(t)\|_{X} dt \\
&\lesssim
\varepsilon^{3/2}+
\varepsilon\mathcal{O}(\varepsilon^{1/4}),
\end{align*} 
which completes (\ref{est:int-R}).
\section{Proof of Theorem \ref{thm:main}
}\label{section:Proof of Theorem}
In this section, we prove Theorem \ref{thm:main}.
Assume that 
$2\le p<3$, $\mathrm{Im}\,\lambda>0$ $\varepsilon>0$ and  
$(1+x^2)\varphi\in\Sigma$. 
Suppose that $u=u(t,x)$ is the time-local solution 
to (\ref{IVP}) 
satisfying $u\in X(T)$ for some $T>0$. 
We immediately see the existence of such $u(t,x)$ 
(see Remark \ref{rem:other def}). 
Recall that positive numbers 
$A$ and $T_B(\varepsilon)$ are given by 
\begin{align*}
A^{-1}
=
(p-1)(\mathrm{Im}\lambda)
\sup_{\xi\in\mathbb{R}}
|\widehat\varphi(\xi)|^{p-1}
\text{\quad{and}\quad}
T_B(\varepsilon)
=
\left(\frac{(3-p)B}{2\varepsilon^{p-1}}
\right)^{2/(3-p)},
\end{align*}
respectively.   
Let $m(t,x)$ and $u_a(t,x)$ be functions defined in
Section \ref{section:Approximate solution}.
For $\varepsilon>0$, 
we fix $\delta=\varepsilon^{1/4}$. 
Then the inequality (\ref{est:u-a}) and 
(\ref{est:int-R}) hold.
We now prove the following lemma:
\begin{lem}\label{lem:key}
Let $2\le p<3$ , $\lambda\in\mathbb{C}$ and $B\in (0,A)$.
Assume that $\mathrm{Im}\,\lambda>0$ and 
$(1+x^2)\varphi\in\Sigma$.
Then there exists some $\varepsilon_0>0$ 
such that the following property holds for any 
$\varepsilon\in (0,\varepsilon_0]$:
If 
\begin{align*}
0<T<\min\{ T(\varepsilon),T_B(\varepsilon)\}
\quad\text{and}\quad
\sup_{0\le t\le T}
\| u_a(t)-u(t)\|_{X}
\le \varepsilon,
\end{align*} 
then we have 
\begin{align*}
\sup_{0\le t\le T}
\| u_a(t)-u(t)\|_{X}
\le \frac{\varepsilon}{2}.
\end{align*} 
\end{lem}
We first prove that 
we obtain Theorem \ref{thm:main}  
by using Lemma \ref{lem:key}.  
We remark that 
the proof is very similar to 
that of \cite{Sunagawa}. 
We fix $B\in (0,A)$ and $\varepsilon\in (0,\varepsilon_0]$.  
Since 
$\| u_a(0)-u(0)\|_X =0$ and 
$\| u_a(t)-u(t)\|_X$ is continuous with respect to $t$, 
there exists some $T^{\ast}>0$ such that 
\begin{align}\label{T ast}
\sup_{0\le t\le T^{\ast}}
\| u_a(t)-u(t)\|_{X}
\le \varepsilon.
\end{align} 
If $T^\ast \ge T_B(\varepsilon)$, 
then it follows from (\ref{est:u-a}) that 
\begin{align}\label{est:a priori}
\sup_{0\le t\le T_B(\varepsilon)}
\| u(t)\|_{X} 
\le 
\sup_{0\le t\le T_B(\varepsilon)}
\| u_a(t)\|_{X}
+
\sup_{0\le t\le T^\ast}
\| u_a(t)-u(t)\|_{X}
\le 
C\varepsilon, 
\end{align} 
where $C$ is a positive constant independent of $\varepsilon$.
The a priori estimate (\ref{est:a priori}) implies that 
$T(\varepsilon)> T_B(\varepsilon)$. 
On the other hand, 
assume that whenever $T^\ast>0$ satisfies (\ref{T ast}), 
$T^\ast$ is smaller than $T_B(\varepsilon)$. 
Then we see that 
\begin{align*}
\sup_{0\le t\le T^\ast}
\| u(t)\|_{X} 
\le 
\sup_{0\le t\le T_B(\varepsilon)}
\| u_a(t)\|_{X}
+
\sup_{0\le t\le T^\ast}
\| u_a(t)-u(t)\|_{X}
\le 
C\varepsilon, 
\end{align*} 
so that 
$T^\ast< \min\{ T(\varepsilon), T_B(\varepsilon) \}$. 
Then a positive number $T^{\ast\ast}$ defined by  
\begin{align*}
T^{\ast\ast} 
=
\max\left\{
T^\ast >0;\  
T^\ast \text{ satisfies (\ref{T ast})} 
\right\}
\end{align*} 
is well-defined and satisfies 
$0<T^{\ast\ast} <\min\{ T(\varepsilon), T_B(\varepsilon)\}$ 
and 
\begin{align}\label{contradiction}
\sup_{0\le t\le T^{\ast\ast}}
\| u_a(t)-u(t)\|_{X}
= \varepsilon
>\frac{\varepsilon}{2}.
\end{align} 
We see from Lemma \ref{lem:key} that 
\begin{align*}
\sup_{0\le t\le T^{\ast\ast}}
\| u_a(t)-u(t)\|_{X}
\le \frac{\varepsilon}{2}, 
\end{align*} 
which contradicts (\ref{contradiction}). 
We hence see that $T_B(\varepsilon)\le T(\varepsilon)$. 
In other words, we have 
\begin{align*}
\varepsilon^{2(p-1)/(3-p)}T(\varepsilon)\ge 
\left(
\frac{(3-p)B}{2} 
\right)^{2/(3-p)}.
\end{align*} 
Since $B\in (0,A)$ is arbitrary, 
Theorem \ref{thm:main} holds. \\

Let us prove Lemma \ref{lem:key}.
Put $v=u_a - u$. 
Then $v$ solves 
\begin{align*}
\left\{
  \begin{array}{ll}
\mathcal{L}Z^\alpha v
=
-
Z^\alpha\mathcal{N}(u_a+v)
+
Z^\alpha\mathcal{N}(u_a)+Z^\alpha R
&\text{in $[0,\infty)\times\mathbb{R}$,}\\
(Z^\alpha v)|_{t=0}=0
&\text{on $\mathbb{R}$.}\\
  \end{array}
\right.
\end{align*}  
The standard energy inequality implies that 
\begin{align}\label{ineq:standard}
\| v(t)\|_{X}
\lesssim
\int_0^t
\| \mathcal{N}(u_a(\tau)+v(\tau))
-
\mathcal{N}(u_a(\tau))\|_{X}
d\tau
+
\int_0^t
\| R(\tau) \|_X
d\tau.
\end{align} 
From (\ref{est:w1-w2}),
the assumption of Lemma \ref{lem:key},
(\ref{est:u-a}) and 
(\ref{est:int-R}), 
we obtain 
\begin{align*}
\| v(t)\|_{X}
\lesssim
\varepsilon^{p-1}
\int_0^t
(1+\tau)^{-(p-1)/2}
\| v(\tau)\|_{X}
d\tau
+
\varepsilon^{3/2}+
\varepsilon\mathcal{O}(\varepsilon^{1/4}).
\end{align*} 
By Gronwall's lemma, 
we see that 
\begin{align*}
\| v(t)\|_{X}
&\le 
C
\left(\varepsilon^{3/2}+
\varepsilon\mathcal{O}(\varepsilon^{1/4})\right)
\exp\left( C
\varepsilon^{p-1}
\int_0^{T_B(\varepsilon)} 
(1+\tau)^{-(p-1)/2}
d\tau
\right) \\
&\le
C
\left(\varepsilon^{3/2}+
\varepsilon\mathcal{O}(\varepsilon^{1/4})\right)
\exp\left( C
\varepsilon^{p-1}
T_B(\varepsilon)^{(3-p)/2}
\right) \\
&\le
C
\left(\varepsilon^{3/2}+
\varepsilon\mathcal{O}(\varepsilon^{1/4})\right)
\exp\left( C
\varepsilon^{p-1}
\varepsilon^{-(p-1)}
\right) \\
&\le 
C\left(\varepsilon^{3/2}+
\varepsilon\mathcal{O}(\varepsilon^{1/4})\right)
\end{align*} 
for some constant $C$ independent of $\varepsilon$.
Choosing $\varepsilon_0>0$ such that 
$C(\varepsilon_0^{1/2}+
\mathcal{O}(\varepsilon_0^{1/4}))\le 1/2$, 
we have $\| v(t)\|_{X}\le \varepsilon/2$,
which completes the proof of Lemma \ref{lem:key}.
Hence Theorem \ref{thm:main} holds.
\subsection*{Acknowledgments}
The author would like to thank 
Professor Hideo Kubo for the helpful advice 
and the referee 
for pointing out some gaps.


\end{document}